\documentclass{amsart}

\usepackage{latexsym}
\usepackage{verbatim}
\usepackage{amsmath}
\usepackage{amssymb}
\usepackage{mathrsfs}
\usepackage{amsthm}
\usepackage{tikz}

\usetikzlibrary{decorations.pathreplacing}
\usetikzlibrary{arrows.meta}

\newcommand{\Ind}{
 \setbox0=\hbox{$x$}\kern\wd0\hbox to 0pt{\hss$
 \mid$\hss}\lower.9\ht0\hbox to 0pt{\hss$\smile$\hss}\kern\wd0
}
\newcommand{\indep}[3]{
 #1\mathop{\mathpalette\Ind{}}_{#2}#3
}

\newcommand{\Notind}{
 \setbox0=\hbox{$x$}\kern\wd0\hbox to 0pt{\mathchardef
 \nn=12854\hss$\nn$\kern1.4\wd0\hss}\hbox to 0pt{\hss$\mid$\hss}\lower.9\ht0
 \hbox to 0pt{\hss$\smile$\hss}\kern\wd0
}

\usepackage{amsthm, amsfonts}
\usepackage{euscript}
\usepackage{graphicx}
\usepackage{float}
\usepackage{caption}
\usepackage{tikz-qtree}
\usepackage{subcaption}

\newcommand{\G}{\Gamma}
\renewcommand{\d}{\delta}
\renewcommand{\a}{\alpha}

\renewcommand{\H}{\mathbb{H}}
\newcommand{\M}{\mathbb{M}}
\newcommand{\inv}{^{-1}}
\newcommand{\C}{\mathcal{C}}
\newcommand{\cl}{\mathrm{cl}}
\newcommand{\Q}{\mathbb{Q}}
\newcommand{\diam}{\mathrm{diam}}

\newtheorem{defi}{Definition}[section]
\newtheorem{theorem}[defi]{Theorem}
\newtheorem{definition}[defi]{Definition}
\newtheorem{lemma}[defi]{Lemma}
\newtheorem{proposition}[defi]{Proposition}

\newtheorem{corollary}[defi]{Corollary}

\newtheorem{remark}[defi]{Remark}
\newtheorem{example}[defi]{Example}
\newtheorem{examples}[defi]{Examples}

\def\Aut{\mathop{\rm Aut}\nolimits}
\def\Isom{\mathop{\rm Isom}\nolimits}

\def\nil2{\mathop{\rm nil2}\nolimits}

\begin{document}
\author{Katrin Tent}
\title{On universal-homogeneous hyperbolic graphs and spaces and their isometry groups}
\begin{abstract}
The Urysohn space is the unique separable metric space that is universal and homogeneous for finite metric spaces, i.e., it embeds any finite metric space  any isometry between finite subspaces extends to an isometry of the whole space. We here consider the existence of a universal-homogeneous hyperbolic space. We show that 
 that for $\delta>0$ there is no $\delta$-hyperbolic space which is universal and homogeneous in the above sense for all finite $\delta$-hpyerbolic spaces. 
 
We then show that for any $\delta\geq 0$ and any countable class $\C$ of $\d$-hyperbolic spaces with countably many distinguished   $\d$-closed subspaces there exists a $\d$-hyperbolic metric space $\H_\C=\H(\C,\delta)$ such that every $X\in \C$ can be embedded into $\H_\C$ as a   $\d$-closed subspace and any isometry between distinguished closed subspaces extends to an isometry of $\H_\C$.

If $\C$ consists of $\d$-hyperbolic geodesic spaces, then $\H_\C$ contains the \emph{quasi-tree of spaces} as defined in \cite{Bestvina}.

For $\C_\d$ the class of all finite $\d$-hyperbolic spaces with rational distances or the class  of finite $\d$-hyperbolic graphs, the limit $\H_\d$ is a $\d$-hyperbolic space (or graph, respectively) universal for all finite $\d$-hyperbolic spaces with rational distances (or finite $\d$-hyperbolic graphs) and such that any isometry between   $\d$-closed subspaces extends to an isometry of $\H_\d$.

We show that the isometry group of $\H_\d$ does not contain elements of bounded displacement and  has no dense conjugacy class.
\end{abstract}

\maketitle
\section{Introduction and non-existence}

The Urysohn space is the unique separable metric space that is universal and homogeneous for finite metric spaces, i.e., it embeds any finite metric space  and any isometry between finite subspaces extends to an isometry of the whole space. The Urysohn space can easily be constructed by amalgamating all finite metric spaces with rational distances and taking the completion of the resulting metric space, see e.g. \cite{TZ}. As is often the case with very homogeneous structures, the isometry group of the Urysohn space has a natural maximal normal subgroup, namely the group of isometries of bounded displacement and the quotient group is boundedly simple. In this note we consider the question whether there is an analog for hyperbolic metric spaces.
This note is motivated by the non-existence result in Theorem~\ref{thm: no universal-homogeneous dh}.

We recall some terminology:
Let $\mathcal{C}$ be a class of structures. Then we say that $\mathcal{C}$ has the

\begin{enumerate}
\item (AP) Amalgamation Property if for any $A, B_1, B_2\in\mathcal{C}$ and embeddings $i_j:A\longrightarrow B_j, j=1,2,$ there is some $D\in\mathcal{C}$ and embeddings $k_j: B_j\longrightarrow D$ such that the diagram commutes, i.e. $i_1\circ k_1=i_2\circ k_2$;
\item (JEP) Joint Embedding Property if any for any $B_1, B_2\in\mathcal{C}$ there is some $D\in\mathcal{C}$ such that $B_2, B_2$ embed into $D$;
\item (HP) Hereditary Property if $\mathcal{C}$ is closed under substructures;

\end{enumerate}

If $\mathcal{C}$ is a countable class of finitely generated structures with (HP), (AP) and (JEP),  there is a unique limit structure $\mathbb{M}$, called the \emph{Fra\"iss\'e limit} of $\mathcal{C}$ (see e.g. \cite{TZmodel theory}, Theorem 4.4.4) characterized by the following  properties:
\begin{enumerate}
\item (universal) $\mathbb{M}$ is universal for $\mathcal{C}$, i.e. every $A\in\mathcal{C}$ embeds into $\mathbb{M}$, and
\item (homogeneous) $\mathbb{B}$ is homogeneous for $\mathcal{C}$, i.e. any isomorphism between substructures $A, B\subset \mathbb{M}$ with $A, B\in\mathcal{C}$ extends to an isomorphism of $\mathbb{M}$.
\item (age) every finitely generated substructure of $\mathbb{M}$ belongs to $\mathcal{C}$.
\end{enumerate}

We record the following well-known observation:
\begin{lemma}\label{lem:homogenous implies amalgamation}
Let $\mathbb{M}$ be a structure which is universal and homogeneous for $\mathcal{C}$ and such that every finitely generated substructure of $\mathbb{M}$ is contained in a structure belonging to $\mathcal{C}$. Then $\mathcal{C}$ has (AP).
\end{lemma}
\begin{proof}
Let $A, B_1, B_2\in\mathcal{C}$ and consider embeddings $i_j:A\longrightarrow B_j, j=1,2,$. Then by universality, $B_1$ and $B_2$ embed into $\mathbb{M}$ and the composition $i_2\circ i_1\inv$ to $A$ induces an isomorphism between the substructures of $B_1$ and $B_2$ isomorphic to $A$. By homogeneity, there is an isomorphism $\alpha$ of $\mathbb{M}$ taking the copy of $A$ inside $B_1$ to the copy of $A$ inside $B_2$. Then any substructure of $\mathbb{M}$ containing $\alpha(B_1)\cup B_2$ is an amalgam as required.
\end{proof}

We will use the following definition of $\delta$-hyperbolicity (see e.g. \cite{DK} Sec. 11.3):

\begin{definition}\label{def:hyperbolic}
Let $(X,d)$ be a metric space, $\d\geq 0$. Then $X$ is $\d$-hyperbolic if and only if for all $x_1, x_2, x_3, x_4\in X$  the following holds: Put
\begin{align*}
E=&d(x_1, x_2)+d(x_3,x_4),\\
 F=&d(x_1, x_3)+d(x_2,x_4),\\ 
 G=&d(x_1, x_4)+d(x_2,x_3).
\end{align*}
\noindent
Then $G\leq \max\{E, F\}+2\d$, or, equivalently, if $E\leq F\leq G$, then $G-F\leq 2\d$.
\end{definition}

Clearly this is a hereditary condition, i.e. if $X$ is $\d$-hyperbolic, then so is any subspace.
It is not too hard to see that for $\d=0$, the class of $0$-hyperbolic (or ultrametric) spaces has (AP), and hence universal-homogeneous ultrametric spaces exist, see e.g. \cite{Ishiki1, Ishiki2, Lemin1, Lemin2, Wan}.
However,  for $\d>0$ there is no universal-homogeneous  $\d$-hyperbolic space:

\begin{theorem}\label{thm: no universal-homogeneous dh}
For $\d>0$ there does not exist a  $\d$-hyperbolic space which is universal and homogeneous for finite $\d$-hyperbolic spaces. 
\end{theorem}

\begin{proof}
By Lemma~\ref{lem:homogenous implies amalgamation} it suffices to show that finite $\d$-hyperbolic spaces with rational distances do not have (AP).

Let $A= \{c,d,e\}$ and $B_1=A\cup\{a\}, B_2=A\cup\{b\}, x_0,x_1,z_0,z_1, z\in\mathbb{R}_{>0}$. 
\begin{center}
\includegraphics[scale=1]{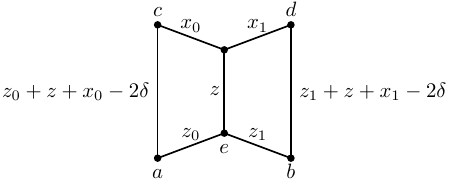}
\end{center}

Suppose that
\begin{align*}
d(c,d) & = x_0 +x_1\\
d(e,c) &= z + x_0\\
d(e,d) &= z + x_1\\
d(a,e) &= z_0,\\
d(a,d) &= z_0 + z + x_1\\
d(a,c) &= z_0 + z + x_0 - 2\delta\\
d(b,e) &= z_1,\\
d(b,d) &= z_1 + z + x_1 - 2\delta\\
d(b,c) &= z_1 + z + x_0.\\
\end{align*}

Then  $\{a,c,d,e\}$ and $ \{b,c,d,e\}$ are $\delta$-hyperbolic.

Clearly $d(a,b) \leq z_0 + z_1.$  If $ z > 2\delta$ and $z_0\neq z_1$, then
\[d(a,d) + d(b,c) = z_0 + z + x_1 + z_1 + z + x_0 > d(a,b) + d(c,d) + 2\delta\]

and similarly
\begin{align*}
d(a,d) + d(b,c) & =  z_0 + z + x_1 + z_1 + z + x_0\\
                & > d(a,c) + d(b,d) + 2\delta\\
               & = z_0 + z + x_0 - 2\delta + z_1 + z + x_1 -2\delta + 2\delta\\
                &= z_0 + z + x_0 + z_1 + z + x_1 +2\delta\\
\end{align*}
So there is no $r\in\mathbb{R}$ such that defining $d(a,b)=r$ would make ${a,b,c,d,e}$ hyperbolic.
\end{proof}

\section{Generalized Fra\"iss\'e limits} \label{sec:Hrushovski limit}

The previous result shows that we cannot construct a universal-homogeneous hyperbolic space.
In order to obtain a universal $\d$-hyperbolic metric space with a weaker homogeneity condition, we modify the amalgamation by restricting to \emph{strong embeddings} to obtain the \emph{Hrushovski limit}, a weaker analog of the Fra\"iss\'e limit.

\medskip\noindent
{\bf{Hrushovski limits:}}  (see e.g. \cite{Hr})  Let $(\mathcal{C},\leq)$ be a countable class of finite structures with a strong embedding relation $\leq$ between structures, defining a well-founded partial order on $\mathcal{C}$. We say that $(\C,\leq)$ is a Hrushovski class if (AP), (JEP) and (HP) hold with respect to strong embeddings, i.e.
\begin{itemize}
\item (AP) for all $A, B_1, B_2\in\mathcal{C}$ with strong embeddings $i_j:A\longrightarrow B_i, i=1,2$, there exists some $D\in\mathcal{C}$ and strong embeddings $k_j:B_j\longrightarrow D$ such that the diagram commutes;
\item (JEP) for all $B_1, B_2\in\mathcal{C}$ there exists some $D\in\mathcal{C}$ and strong embeddings $k_j:B_j\longrightarrow D, j=1,2$;
\item (HP) $\C$ is closed under strongly embedded substructures.
\end{itemize}
 In this case there is a limit structure $\mathbb{M}$, the Hrushovski limit for  $(\mathcal{C},\leq)$. For $A\in\mathcal{C}$ we define $A\leq \mathbb{M}$ if $A\leq C$ for any $C\in\mathcal{C}$ contained in $\mathbb{M}$ with $A\subseteq C$. The Hrushvoski limit is universal for $\mathcal{C}$, i.e. every $A\in\mathcal{C}$ strongly embeds into $\mathbb{M}$ and $\mathbb{M}$ is $\leq$-homogeneous, i.e. any isomorphism between strongly embedded substructures extends to an isomorphism of $\mathbb{M}$.

We extend these amalgamation constructions to a more general setting:

\begin{definition}\label{def:marked structure}
A \emph{marked structure} is a structure $A$ with an (at most) countable class of distinguished substructures\footnote{There is no assumption on the class of distinguished substructures or on the distinguished substructures themselves to be first-order definable. } closed under intersections.

An \emph{isomorphism} between marked structures $A, B$ is given by an isomorphism between the structures $A$ and $B$ inducing isomorphisms on all distinguished substructures (and, in particular, a bijection between the sets of distinguished substructures).

An \emph{embedding} of a marked structure $A$ into a marked structure $B$ is given by an embedding $f$ of $A$ into $B$  such that $f(A)$ is a distinguished substructure of $B$ and $f$ induces an isomorphism of the marked structures $A$ and $f(A)$, where $f(A)$ is the marked structure obtained from $B$ by restricting the distinguished substructures of~$B$.

\end{definition}

\begin{definition}\label{def:marked Fraisse class}
We say that a countable class $\mathcal{C}$ of marked structures is a marked Fra\"iss\'e class if $\mathcal{C}$ satisfies the following conditions:

\begin{enumerate}
\item (HP) $\mathcal{C}$ is closed under distinguished substructures, i.e. if $A\in\mathcal{C}$ is a marked structure and $B$ is a distinguished substructure in $A$, then $B\in\mathcal{C}$.
\item (JEP) $\mathcal{C}$ is closed under joint embeddings, i.e. for $A, B\in\mathcal{C}$ there is some $D\in\mathcal{C}$ such that $A, B$ embed into $D$ as marked structures.
\item (AP) $\mathcal{C}$ is closed under amalgamation, i.e. if $A, B,C\in\mathcal{C}$ with embeddings of $A$ into $B$ and $C$, there is some $D\in\mathcal{C}$ and embeddings of $B, C$ into $D$ such that the corresponding diagram commutes.

\end{enumerate}

\end{definition}

The following  examples of marked Fra\"iss\'e classes show that this notion generalizes both Fra\"iss\'e classes and Hrushovski classes:
\begin{examples}\label{examples marked}

\begin{enumerate}
\item If $\mathcal{C}$ is a Fra\"iss\'e class of finitely generated structures over a countable language, then $\C$ is a marked Fra\"iss\'e class by considering all finitely generated substructures of elements in $\C$ as distinguished substructures.

\item If $(\C,\leq)$ is a Hrushovski class, then $\C$ is a marked Fra\"iss\'e class by considering all strongly embedded substructures of elements in $\C$ as distinguished substructures.

\item Let $T$ be the infinite tree with countably infinite valencies considered as a geodesic metric space with edges of length $1$. Then the class $\C$ of all subtrees of $T$ of finite diameter with distinguished substructures given by subtrees  is a marked Fra\"iss\'e class.
\end{enumerate}

\end{examples}

The construction of Fra\"iss\'e limits given e.g. in the proof of \cite{TZmodel theory} Thm. 4.4.4 immediately generalizes to this context and yields the following result:

\begin{theorem}\label{thm:generalized Fraisse}
Let $\C$ be a marked Fra\"iss\'e class.
Then there exists a unique marked structure $\mathbb{M}$  with the following properties:
\begin{enumerate}
\item (age) every distinguished substructure of $\mathbb{M}$ belongs to $\C$;
\item (cofinal) $\M$ is the union of a countable chain of distinguished substructures;
\item (universal) every $A\in\C$ embeds into $\mathbb{M}$; and
\item (homogeneous) if $f$ is an isomorphism of distinguished substructures $A, B$ of $\mathbb{M}$, then there is an automorphism of $\mathbb{M}$ extending $f$.
\end{enumerate}

\end{theorem}

We will use this generalization of Fra\"iss\'e's theorem to amalgamate countable classes of more general (not necessarily countable) structures. None of the conditions in this process depend on a particular first-order language and so we generally deal with abstract metric spaces (with suitable additional properties) instead.

\section{Strong embeddings of $\d$-hyperbolic spaces}\label{sec:strong embeddings}

We define a property of subsets of metric spaces inspired by buildings:

\begin{definition}\label{def:gated}
Let $(X,d)$ be a metric space. We call a subset $A\subseteq X$ \emph{gated} in $X$ if for every $b\in X\setminus A$ there exists a \emph{gate} $g=g_A(b)\in A$  such that for all $a\in A$ we have $d(b,a)=d(b,g)+d(g,a)$. We call $g$ the gate for $b$ in $A$.
\end{definition}

\begin{remark}\label{rem:unique gate}
Note that if $A$ is gated in $X$, then for $b\in X\setminus A$ the gate $g_A(b)$ is uniquely determined: if $g$ and $g'$ are gates for $b$  in $A$, then $d(b, g)=d(b,g')+d(g',g)$ and $d(b,g')=d(b,g)+d(g,g')$ forcing $d(g,g')=0$. In particular, $g_A(b)$ is the unique element in $A$ with minimal distance to $b$.
\end{remark}

\begin{example}\label{ex:cycle}
Clearly, any convex set in a tree is gated. On the other hand, 
suppose that $X=(x_0, x_1,\ldots, x_n=x_0)$ is a cycle of length $n$ in a graph with the graph metric. Then the subset $X=\{x_1,\ldots, x_{n-1}\}$ is not gated in $X$ as $d(x_0,x_1)=d(x_0,x_{n-1})=1$, contradicting our previous remark.
\end{example}

For a metric space $(X,d)$ we call a tuple $(x_0,\ldots, x_n)$ \emph{geodetic} if $d(x_0,x_n)=d(x_0,x_1)+\ldots+d(x_{n-1}, x_n)$. The following is an easy, but useful consequence of the definition:
\begin{remark}\label{rem:geodetic}
 If $A$ is gated in $B$, then $A$ is \emph{convex} in $B$ in the following sense: if $(a_1,b,a_2)$ is geodetic with $a_1, a_2\in A, b\in B$, then $b\in A$.
\end{remark}

We next define a notion of strong embeddings for metric spaces that will be appropriate for classes of hyperbolic metric spaces to allow amalgamation:

\begin{definition}\label{def:closed}
Fix $\d\geq 0$ and let $X$ be a metric space, $A\subset X$. Then we say that $A$ is $\d$-\emph{closed} (or: strongly embedded)   in $X$ and write $A\leq_\d B$ (or just $A\leq B$ if $\d$ is clear from the context) if
\begin{enumerate}
\item $A$ is gated in $B$;
\item for $b,b'\in B\setminus A$ with $d(g_A(b),g_A(b'))>\d$ we have \[d(b,b')=d(b,g_A(b))+d(g_A(b),g_A(b'))+d(g_A(b'),b').\]
\end{enumerate}
\end{definition}
Clearly, any gated set is closed in the sense of the topology induced from the metric. Note however that the empty set is not gated and hence never   $\d$-closed.

\begin{remark}\label{rem:closed}
It is easy to see that being   $\d$-closed is transitive, i.e., if $A\leq_\d B$ and $B\leq_\d C$, then $A\leq_\d C$.
We also note that single points are   $\d$-closed in any metric space.
\end{remark}

\begin{lemma}\label{lem:intersection}
Let $A, B$ be $\d$-closed subsets of a  metric space $X$ with $A\cap B\neq\emptyset$.  Then for $b\in B\setminus A$ we have $g_A(b)\in A\cap B$.
\end{lemma}
\begin{proof}
 Let $c\in A\cap B$.
Then $(b,g_A(b),c)$ is geodetic with $b,c\in B$. By Remark~\ref{rem:geodetic} we have   $g_A(b)\in B\cap A$.
\end{proof}

\begin{corollary}\label{cor:submodular}
If $A, B$ are $\d$-closed subsets of a  metric space $X$ with $A\cap B\neq\emptyset$, then $A\cap B\leq_\d B\leq_\d X$. 
\end{corollary}
\begin{proof}
Since being   $\d$-closed is transitive by Remark~\ref{rem:closed}, it suffices to show that $A\cap B$ is  $\d$-closed in $A$ and in $B$ and this follows from Lemma~\ref{lem:intersection}. 
\end{proof}

\begin{corollary}\label{cor:closure}
Any nonempty subset $A$ of a metric space $X$ is contained in a unique smallest   $\d$-closed set $\cl(A)=\bigcap \{B\leq_\d X\colon A\subset B\}$.
\end{corollary}

By Remark~\ref{rem:geodetic} any   $\d$-closed set is convex. However, the converse does not hold:

Suppose $x, y\in X$ and there is some $z\in X$ such that $\{x,y,z\}$ is a non-degenerate triangle. If there is no $z'$ such that both $(z,z',x)$ and $(z,z',y)$ are geodetic,  then $z\in\cl(\{x,y\})$.

Note also that the union of  $\d$-closed subsets $A, B$ need not be  $\d$-closed. However, we have the following partial result which will be used in Section~\ref{sec:isomgroup}.

\begin{lemma}\label{lem:strong union}
Suppose $A, B$ are  $\d$-closed subsets of a metric space $X$, $A\cap B\neq\emptyset$ and $B_\d(A\cap B)\cap (B\cup A)=B\cap A$ where $B_\d(A\cap B)$ denotes the ball of radius $\d$ around $A\cap B$. Then $A\cup B$ is  $\d$-closed in $X$.
\end{lemma}

\begin{proof}
We first verify condition (1) of Definition~\ref{def:closed}, i.e. we show that $A\cup B$ is gated in $X$. Let $x\in X\setminus A\cup B$ and let $g_1=g_A(x), g_2=g_B(x)$.
 By symmetry we may assume $d(x,g_1)\leq d(x,g_2)$. 
If $g_2\in A\cap B$, then $d(x,g_2)=d(x,g_1)+d(g_1,g_2)$ and hence $g_1=g_{A\cup B}(x)$ is the required gate.

 So suppose towards a contradiction that $g_2\in B\setminus A$. Then $c=g_A(g_2)\in A\cap B$ by Lemma~\ref{lem:intersection}. By assumption $B_\d(A\cap B)\cap (B\cup A)=B\cap A$, so $d(g_1,c)>\d$ and since $A$ is $\d$-closed it follows from the second condition of Definition~\ref{def:closed} that
\[d(x,g_2)=d(x,g_1)+d(g_1,c)+d(c,g_2).\]
Hence $c=g_2$, contradicting our assumption.

We now verify condition (2) of Definition~\ref{def:closed}: let $x,y\in X$, $c=g_{A\cup B}(x),d=g_{A\cup B}(y)$ and assume $d(c,d)>\delta$. We have to show that 
\[d(x,y)=d(x,c)+d(c,d)+d(d,y).\]
 If $c,d\in A$ or $c,d\in B$, this follows from the assumption that $A, B$ are $\d$-closed.
 So suppose by symmetry that $c\in A\setminus B, d\in B\setminus A$ and let $e=g_A(y)$. Then by the first part of the proof, we have $e=g_A(d)\in A\cap B$ and hence $d(c,e)>\d$. Since $A$ is $\d$-closed, we have
$d(x,y)=d(x,c)+d(c,e)+d(e,y)$ as required.
\end{proof}

We first define a canonical amalgam for arbitrary metric spaces:
\begin{definition}\label{def:canonical amalgam}
Let $A, B_1, B_2$ be metric spaces with $A\subseteq B_1,B_2$. Then 
the \emph{canonical amalgam} $D=B_1\otimes_AB_2$ of $B_1$ and $B_2$ over $A$  is defined as the disjoint union of $A$ with $B_1\setminus A$ and $B_2\setminus A$ with the metric extending the metric on $B_1, B_2$ defined as follows: for $x\in B_1\setminus A, y\in B_2\setminus A$ we put 
\[d(x,y)=d(x, g_A(x))+d(g_A(x),g_A(y))+d(g_A(y),y).\] 
\end{definition}
The crucial lemma is the following:
\begin{lemma}\label{lem:amalgamation}

Let $A, B_1, B_2$ be $\d$-hyperbolic spaces and assume that $A\leq_\d B_1,B_2$. Then 
the canonical amalgam $D=B_1\otimes_AB_2$ of $B_1$ and $B_2$ over $A$ is $\d$-hyperbolic and $B_1,B_2\leq_\d D$.
\end{lemma}
\begin{proof}
It is clear from the definition of the canonical amalgam that $B_1, B_2\leq_\d D$. Hence it suffices to show that $D$ is $\d$-hyperbolic. 
 Let $x_1,\ldots, x_4\in D$. 
If $x_1,\ldots, x_4$ are contained in $B_1$ or $B_2$, there is nothing to show. 
We now write $g_i=g_A(x_i)$ if $x_i\notin A$ and $g_i=x_i$ if $x_i\in A$

\medskip\noindent
{\bf Case 1:} Assume $x_1,x_2,x_3\in B_1, x_4\in B_2\setminus A$.

Consider 
\begin{align*}
E=& d(x_1, x_2)+d(x_3,x_4)=&d(x_1,x_2)+d(x_3,g_3)+d(g_3,g_4)+d(g_4,x_4),\\
 F=& d(x_1, x_3)+d(x_2,x_4)=&d(x_1, x_3)+d(x_2,g_2)+d(g_2,g_4)+d(g_4,x_4)\\
 G=& d(x_1, x_4)+d(x_2,x_3)=&d(x_1,g_1)+d(g_1,g_4)+d(g_4, x_4)+d(x_2,x_3)
\end{align*}

Each of $E, F, G$ contains the summand $d(g_4,x_4)$. Subtracting this summand from each of $E, F, G$, we get the distances corresponding to the points $x_1,x_2,x_3,g_4\in B_1$. Since $B_1$ is hyperbolic, these points satisfy the condition in Definition~\ref{def:hyperbolic}, and hence so do $x_1,x_2,x_3, x_4$.

\medskip\noindent
{\bf Case 2:} Assume $x_1,x_2\in B_1, x_3,x_4\in B_2\setminus A$
 and let
\begin{align*}
E_1&=d(g_1,g_2)+d(g_3,g_4)\\
F_1&=d(g_1,g_3)+d(g_2,g_4)\\
G_1&=d(g_1,g_4)+d(g_2,g_3).\\
\end{align*}
Then
\begin{align*}
E&=d(x_1, x_2)+d(x_3,x_4)\leq & E_1+\sum_{i=1}^4 d(x_i,g_i) \\
 F&=d(x_1, x_3)+d(x_2,x_4)= & F_1+\sum_{i=1}^4 d(x_i,g_i) \\
 G&=d(x_1, x_4)+d(x_2,x_3)= & G_1+\sum_{i=1}^4 d(x_i,g_i).
\end{align*}

If  $d(g_3, g_4), d(g_1,g_2)>\d$, then $E=E_1+\sum_{i=1}^4 d(x_i,g_i)$. Cancelling $\sum_{i=1}^4 d(x_i,g_i)$  in each of $E,F, G$ we obtain the hyperbolicity condition for the points $g_1,\ldots, g_4\in~A$. Since $A$ is hyperbolic, the result follows.

Now assume (by symmetry) $d(g_3,g_4)\leq \d$. Then by the triangle inequality we have
\begin{align*}
|d(g_1,g_3)-d(g_1,g_4)|\leq& \d \mbox{  and } \\
|d(g_2,g_4)-d(g_2,g_3)|\leq &\d.
\end{align*}
 Hence
\begin{equation}\label{eq: amalgam}
|F_1-G_1|=|F-G|\leq 2\d.
\end{equation}

If $E\leq F\leq G$ or $F\leq E\leq G$,  the claim follows from (\ref{eq: amalgam}).

\medskip
So suppose $F\leq G<E$. Since
\[d(g_1,g_2)\leq d(g_1,g_4)+d(g_2,g_3)+\d=G_1+\d\]
we have $E_1\leq G_1+2\d$. Hence 
\[G<E\leq E_1+ \sum_{i=1}^4 d(x_i,g_i) \leq G_1+ \sum_{i=1}^4 d(x_i,g_i)+2\d= G+2\d\] and so $E-G\leq 2\d$ as required.

\end{proof}

We now specialize marked classes to $\d$-hyperbolic spaces:
\begin{definition}\label{def:marked delta class}
A \emph{marked class} of $\d$-hyperbolic spaces is 
a countable class $\C$ of $\d$-hyperbolic spaces such that all distinguished subspaces of $A$ are $\d$-closed in $A$ and $A$ contains at least one distinguished subspace consisting of a single point.
\end{definition}

Thus, by definition distinguished subspaces are $\d$-closed, but in general not all $\d$-closed subspaces are distinguished in a marked class of $\d$-hyperbolic spaces.

\begin{theorem}\label{thm:amalgamation works}
Fix $\d$ and let $\C_0$  be a marked class of $\d$-hyperbolic spaces.
Then there is a marked Fra\"iss\'e class $\C$ of  $\d$-hyperbolic spaces containing $\C_0$. Thus, $\C$ has a unique marked Fra\"iss\'e limit which is universal and homogeneous for
 distinguished  subspaces and any isometry between distinguished subspaces of $\H_\C$ extends to an isometry of $\H_\C$.
\end{theorem}

\begin{proof}
Since any $A\in\C$ contains a distinguished substructure consisting of a point, we obtain (JEP) by identifying distinguished points and amalgamating over this point. Hence it suffices to verify (AP).

We build the class $\C$ inductively from $\C_0$ as follows: suppose at stage $i$ we have constructed a marked class $\C_i$  of $\d$-hyperbolic spaces with countably many distinguished $\d$-closed subspaces. Then we define 
\[C_{i+1}=\C_i\cup \{B_1\otimes_A B_2\colon A\leq B_1, B_2\in C_i, A \mathrm{\ distinguished\ in\ } B_1, B_2\}. \] 
Furthermore, for any $D=B_1\otimes_A B_2$ with $A\leq B_1,B_2\in\C_i$  we define the distinguished subspaces as the  $\d$-closed subspaces $X$ of $D$ such that $X\cap B_1$ and $X\cap B_2$ are  distinguished in $B_1, B_2$, respectively.
Then $\C=\bigcup_{i\in\mathbb{N}}\C_i$ has (AP) and (JEP). 
Hence the marked Fra\"iss\'e limit exists and is as desired.
\end{proof}

The previous theorem implies the  existence of a  universal and $\leq$-homogeneous $\d$-hyperbolic space~$\H_\d$ for the class of finite $\d$-hyperbolic spaces with rational distances. (In fact, this class is a Hrushvoski class with respect to the strong embeddings defined above.)  More generally, we may allow distance taking values in any countable subsemigroup $(\mathbb{R}_{\geq 0},+)$. However, for simplicity of notation we restrict to the case of rational distances.

\begin{corollary}\label{cor: finite hyperbolic class}
Fix $\d\geq 0$ and let $\C_\d$ be the class of finite $\d$-hyperbolic spaces with rational distances. Since $\C_\d$ is countable and contains only countably many $\d$-closed subsets, we can consider this class as a marked class with all  $\d$-closed subspaces as distinguished substructures. Then the marked Fra\"iss\'e limit $\H_\d$ exists and is homogeneous and universal for finite $\d$-closed subspaces.
\end{corollary}
\begin{proof}
By Lemma~\ref{lem:amalgamation} the marked class $\C_\d$ has (AP) and (JEP). Since every finite $\d$-closed subset of $A\in\C_\d$ is a distinguished substructure, the same holds for the marked Fra\"iss\'e limit $\H_\d$.  Hence $\H_\d$ is homogeneous and universal for finite $\d$-closed subspaces.
\end{proof}

The same reasoning also applies to the class of finite $\d$-hyperbolic graphs with the graph metric considered as a marked class with all finite $\d$-closed  induced subgraphs as distinguished substructures. Note that since $\d$-closed sets are convex, the graph distance on a $\d$-closed subgraph agrees with the graph distance in the ambient graph. Hence we obtain:

\begin{corollary}\label{cor: finite hyperbolic graphs}
For  any $\d\geq 0$ there is a countable $\d$-hyperbolic graph $\G_\d$  universal and homogeneous for finite $\d$-hyperbolic graphs.
\end{corollary}

Note that $\H_\d$ (and of course $\G_\d$) are discrete $\d$-hyperbolic spaces. However, we can also obtain non-discrete or geodesic homogeneous $\d$-hyperbolic spaces as the next examples show:

\begin{examples}\label{ex: dense set}

$\mathrm{(i)}$  Let $T$ be a regular tree with vertices of valency $k\geq 3$. Note that the set of vertices of $T$ with the graph metric is a countable $0$-hyperbolic space. Since every infinite ray in $T$ is a  $\d$-closed subset of $T$, we see that a countable metric space may have uncountably many  $\d$-closed subspaces.

$\mathrm{(ii)}$ Let $\C_0$ consist of the rationals as a metric space. Since there are only countably many  $\d$-closed subsets in the rationals, we can take all  $\d$-closed subsets as distinguished subspaces. This satisfies the assumptions of Theorem~\ref{thm:amalgamation works} and $\H_\C$ will be a tree infinitely branching at every point with all bi-infinite paths isometric to the rationals and its completion is an $\mathbb{R}$-tree.

$\mathrm{(iii)}$  Let $\C_0$ be the real line as a metric space with a countable collection  of  $\d$-closed  intervals as distinguished subspaces. Then the marked Fra\"iss\'e limit is again an $\mathbb{R}$-tree infinitely branching at countably many points.
\end{examples}

\section{Projection subspaces of $\H_\C$}\label{sec:projection}

In \cite{Bestvina} the authors show how to assemble a class $\bf{Y}$ of geodesic spaces to a quasi-tree. The purpose of this section is to show that the quasi-tree of spaces they construct can be retrieved in the corresponding universal space $\H_\C$, at least if the class $\bf{Y}$ is countable.

We write
\[d(X,Y)=\inf\{d(x,y)\colon x\in X, y\in Y\}.\] 

For   $\d$-closed subsets, this infinimum is attained:

\begin{lemma}\label{lem: projections}
Let $A, B$ be   $\d$-closed and disjoint subsets of a metric space $X$ and let $G_A(B)=\{g_A(b)\colon b\in B\}$ and $G_B(A)=\{g_B(a)\colon a\in A\}$.
Then \[\diam(G_A(B)), \diam(G_A(B))\leq\d.\] 
Furthermore, for all $a\in G_A(B), b\in G_B(A)$ we have  \[d(a,g_B(a))=d(b,g_A(b))=d(A,B)\] and there is an isometry \[g:G_A(B)\longrightarrow G_B(A),\quad a\mapsto g_B(a).\]

\end{lemma}
\begin{proof}
We first claim that $G_A(B)=\{g_A(b)\colon b\in G_B(A)\}$ and $G_B(A)=\{g_B(a)\colon a\in G_A(B)\}$: let $b\in B$ and $a=g_A(b)$.
Then $(a,g_B(a),b)$ is geodetic and $a=g_A(g_B(a))$.

Now let $a, a'\in G_A(B)$, say $a=g_A(b),a'=g_A(b'), b,b'\in G_B(A)$. Then $b=g_B(a), b'=g_B(a')$.

If $d(a,a')>\d$, we have $d(b,b')=d(b,a)+d(a,a')+d(a',b')$ since $A$ and $B$ are  $\d$-closed. Hence $d(b,b')>\d$ and thus also $d(a,a')=d(a,b)+d(b,b')+d(b',a')$, a contradiction.

Now suppose $d(a,b)<d(a',b')$. By symmetry we may assume $d(a,a')\leq d(b,b')$. Then since $b'=g_B(a')$ we have $d(a',b)=d(a',b')+d(b',b)>d(a',a)+d(a,b)$, a contradiction. Hence, $d(a,b)=d(a',b')$ is constant and equal to $d(A,B)$.

Finally suppose $d(a,a')<d(b,b')$. Then as before 
\[ d(b,a')=d(b,a)+d(a,a')=c+d(a,a')<d+d(b,b')=d(a',b')+d(b,b')=d(a',b),\]
 a contradiction. Thus, the gate map is an isometry from $G_A(B)$ to $G_B(A)$.
\end{proof}

\begin{remark}\label{rem:gate set not closed}
Note that if $A, B$ are closed and disjoint, then the gate sets $G_A(B)$ and $G_B(A)$ need not be   $\d$-closed: let $A=\{x_0,x_1,x_2\}, B=\{y_0,y_2\}$ with $d(x_0,y_0)=d(x_2,y_2)=d(x_0,x_2)=d(y_0,y_2)=\d$ and $d(x_1,x_0)=1/3\d, d(x_1,x_2)=2/3\d$. Then $A, B\leq A\cup B$, but $G_A(B)=\{x_0,x_2\}$ is not   $\d$-closed.
\end{remark}

We now introduce the projection metric on   $\d$-closed subsets from \cite{Bestvina}.

\begin{definition}\label{def:proj distance}\rm{[see \cite{Bestvina}, 4.1] }
Let $X, Y, Z$  be pairwise disjoint gated subsets of a metric space. We put
\[d^\pi_Y (X, Z) = \diam (G_Y(X) \cup G_Y(Z)).\]
\end{definition}

\begin{lemma}\label{lem:closed sets between}
Let $X, Y, Z$  be pairwise disjoint   $\d$-closed subsets of a metric space and suppose  $d^\pi_Y (X, Z)> \d$. Then
\[d(X,Z)>d(X,Y)+d(Y,Z)+\d.\]
\end{lemma}
\begin{proof}
 Since by Lemma~\ref{lem: projections} each set $G_Y(X)$ has diameter $\leq \d$, there exist $x\in X, z\in Z$ such that $d(g_Y(x),g_Y(z))>\d$. Since $Y$ is   $\d$-closed, this implies that the sequence $(x,g_Y(x),g_Y(z),z)$ is geodetic and so $x\in X$ and $g_Y(z)\in Y$ have the same gate in $Z$. Thus the claim follows from  Lemma~\ref{lem: projections}.
\end{proof}

\begin{proposition}\label{prop:projection axioms}
Let $\bf{C}$ be a class of pairwise disjoint   $\d$-closed subsets of a metric space, $\d>0$, and $X, Y, Z, W\in \bf{C}$. Then the projection complex axioms of \cite{Bestvina}  hold, i.e. we have the following:

\begin{itemize}
\item[(PC 1)]  $d^\pi_Y (X, Z)=d^\pi_Y (Z, X)$;
\item[(PC 2)]  $d^\pi_Y (X, Z)+ d^\pi_Y (Z, W)\geq d^\pi_Y (X, W)$;
\item[(PC 3)]  $\min\{d^\pi_Y (X, Z), d^\pi_Z (X, Y)\}\leq \d$;
\item[(PC 4)]  for all $X, Z\in\bf{C}$ the number of $Y\in\bf{C}$ such that $d^\pi_Y (X, Z)>\d$ is finite.
\end{itemize}
\end{proposition}
\begin{proof}
Clearly (PC 1) holds and (PC 2) follows directly from the triangle inequality in metric spaces.

For (PC 3) suppose $X, Y, Z$  are pairwise disjoint and   $\d$-closed sets with both 
\[d^\pi_Y (X, Z)>\d \mbox{\ and\ } d^\pi_Z (X, Y)> \d.\] Then by Lemma~\ref{lem:closed sets between} we have  \[d(X,Z)>d(X,Y)+d(Y,Z)+\d \mbox{\ and \ }d(X,Y)>d(X,Z)+d(Z,Y)+\d,\] which is impossible.

For (PC 4) suppose that $X, Y_1, Y_2, Z$  are pairwise disjoint   $\d$-closed sets with both $d^\pi_{Y_1} (X, Z), d^\pi_{Y_2} (X, Z)> \d$.
Since all gate sets have diameter at most $\d$, we cannot have both
$d^\pi_{Y_1}(X,Y_2)\leq \d$ and $d^\pi_{Y_1}(Z,Y_2)\leq \d$. Now if $d^\pi_{Y_1}(X,Y_2)> \d$, then by Lemma~\ref{lem:closed sets between} we have
\[d(X,Y_2)>d(X,Y_1)+d(Y_1,Y_2)+\d \mbox{ and } d(X,Z)>d(X,Y_2)+d(Y_2,Z)+\d.\]

Thus, $d(X,Z)>d(X,Y_1)+d(Y_1,Y_2)+d(Y_2,Z)+2\d$.
Hence the number of $Y\in\bf{C}$ such that $d^\pi_Y (X, Z)>\d$ is bounded by $d(X,Z)/\d$.
\end{proof}

\section{Stationary independence relations}\label{sec:SIR}

In this section we verify that each of the universal $\leq$-homogeneous spaces $\H_\C$ constructed in the previous section supports a stationary independence relation, which will be useful in the study of their isometry groups (see \cite{TZ}).

Since marked subspaces are closed under intersections, we have the following stronger form of Corollary~\ref{cor:closure}:
\begin{corollary}\label{cor: marked closure}
Any non-empty finite subset $A$ of $\H_C$ is contained in a unique smallest   $\d$-closed marked set $\cl(A)=\bigcap \{B\leq_\d X\colon A\subset B\}$.
\end{corollary}

\begin{definition}\label{def:SIR}
For finite subsets $A, B, C\subseteq \H_\C, B\neq\emptyset,$ we say that $A$ is independent from $C$ over $B$, $\indep{A}{B}{C}$,
if and only if $\cl(A\cup B\cup C)$ is isometric to the canonical amalgam of $\cl(A\cup B)\otimes_{\cl(B)}\cl(B\cup C)$.
\end{definition}

Note: the sets  $\cl(A\cup B\cup C),  \cl(A\cup B)$ and  $\cl(A\cup C)$  are marked subsets of $\H_\C$. The canonical amalgam  $\cl(A\cup B)\otimes_{\cl(B)}\cl(B\cup C)$ is an element in $\C$. So $A$ and $C$ are independent over $B$ if the distances in $\cl(A\cup B\cup C)$ agree with the distances in the canonical amalgam. 

\begin{proposition}\label{prop:SIR}

$\indep{}{}{}$ defines a local stationary independence relation on $\H_\C$, i.e. the following conditions are satisfied for all finite sets $A, B, C, D\subset\H_\C$ where $B\neq\emptyset$:

\begin{enumerate}
\item $\indep{A}{B}{C}$ is invariant under automorphisms of $M$;
\item 
\[\indep{A}{B}{CD} \mbox{\ if and only if \ }\indep{A}{B}{C} \mbox{ and \ } \indep{A}{BC}{D};\]

\item $\indep{A}{B}{C}$ if and only if $\indep{C}{B}{A}$;
\item there is some finite set $A'$ such that $\cl(A\cup B)$ and $\cl(A'\cup B)$ are isometric and $\indep{A'}{B}{C}$;
\item if $\cl(A\cup B)$ and $\cl(D\cup B)$ are isometric and both $A$ and $D$ are independent from $C$ over $B$, then $\cl(A\cup B\cup C)$ and $\cl(D\cup B\cup C)$ are isometric.
\end{enumerate}
\end{proposition}
\begin{proof}
All conditions follow easily from the $\leq$-homogeneity and universality of $\H_\C$.
\end{proof}

\section{The isometry group of $\H_\d$}\label{sec:isomgroup}

We now turn our attention to $\H_d$, the unique $\d$-hyperbolic space universal and $\leq_\d$-homogeneous for finite $\d$-hyperbolic spaces with distances in $\Q$.

Its isometry group $G=\Isom(\H_\d)$ is a Polish group, where a neighbourhood basis of the identy is given by point stabilizers of finite   $\d$-closed subsets of $\H_\d$. Since every element of $\H_\d$ is   $\d$-closed in $\H_\d$, it follows that $G$ acts transitively on $\H_\d$.

In the case of the Urysohn space $\mathbb{U}$ it is easy to see that the isometry group acts primitively: suppose there is an invariant equivalence relation $E$ with $E(x,y)$ for some $x,y\in\mathbb{U}$ and let $z\in\mathbb{U}$. Then by universality there is a sequence $(x_0=x,x_1=y,\ldots, x_n=z)$ for some $n\in\mathbb{N}$ such that $d(x_i,x_{i-1})=d(x_1,x_0), i=1,\ldots n$. By homogeneity and invariance of $E$ we see that $E(x,z)$ showing that the equivalence class is all of $\mathbb{U}$.

It is also easy to see that for any regular tree $T$ the isometry group respects a bipartition of the vertices and acts primitively on each class. Thus, a 2-coloring is the only invariant equivalence relation on $T$.
In contrast to this, we will now show that $\H_\d$ has non-trivial invariant equivalence relations, in fact uncountably many:

To show this, we introduce the following auxilliary graph structure on $\H_\d$. Let $R\subset\mathbb{Q}_{> \d}$. We define edges $E_R$ on $\H_\d$ in the following way: for $x,y\in\H_\d$ we put $E_R(x,y)$ if and only if $d(x,y)\in R$ and $\{x,y\}\leq \H_\d$.
Note that this edge relation is invariant under the isometry group of $\H_\d$.\footnote{The definition of the graph structure on $\H_\d$ uses only the metric. This is not an extension of a first-order language.}

\begin{lemma}\label{lem:tree}
For any $R\subset\mathbb{Q}_{> \d}$, the graph $(\H_\d,E_R)$ is a forest, i.e. the connected components of this graph are regular trees of infinite valencies and any connnected subgraph is   $\d$-closed in the sense of the metric on $\H_\d$.
\end{lemma}
\begin{proof}
The fact that every vertex in this graph has infinite valency follows directly from the universality of $\H_\d$: clearly,  the set $X=\{x,y_1,\ldots, y_n\}$ with $d(x,y_i)\in R$ and $ (y_i,x,y_j)$ geodetic for $ 1\leq i\neq j\leq n$ can be strongly embedded into $\H_\d$ for any $n\in\mathbb{N}$. Then $\{x,y_i\}\leq X\leq \H_\d$ for any $i\leq n$ and hence $E_R(x,y_i)$ holds. Since the isometry group acts transitively on the points of $\H_\d$, it follows that every point has infinite valency in this graph.

By Lemma~\ref{lem:strong union} it follows inductively, that any finite connected subgraph of $(\H_\d,E_R)$ that does not contain a cycle is   $\d$-closed in $\H_\d$: clearly this holds for any edge of the graph. So suppose inductively that $A$ is a finite subgraph without cycles, let $x\in A$ and extend $A$ by an edge $E_R(x,y)$. Then the conditions of Lemma~\ref{lem:strong union} are satisfied, and hence  $A\cup\{y\}\leq\H_\d$. Furthermore, if $A$ is a finite connected subgraph of $(\H_\d,E_r)$ without cycles, then for $x,y\in A$ the unique path in $A$ from $x$ to $y$ is geodetic, i.e. the distance between points in $A$ in the metric of $\H_\d$ is the graph distance in $A$ where an edge $E_R(x,y)$  has length $r$ if $d(x,y)=r\in R$.

Now suppose that $(x_0,\ldots, x_n=x_0)$ is a   $\d$-closed cycle in this graph with $n~>~2$. Then the subpath $X_0=\{x_1,x_2, x_{n-1}\}$ is geodetic and strong in $\H_\d$ by the previous paragraph. Thus, by  Remark~\ref{rem:geodetic}, $g_{X_0}(x_0)\in \{x_{n-1}, x_1\}$. However, if $x_1=g_{X_0}(x_0))$, then $(x_0,\ldots, x_{n-1})$ is geodetic. By  Remark~\ref{rem:geodetic}, the set $\{x_0, x_{n-1}\}$ is not   $\d$-closed in $\H_\d$ and so we cannot have $E_R(x_{n-1},x_0)$.  The symmetric argument applies if $x_{n-1}=g_{X_0}(x_0))$. Thus we cannot have both $E_R(x_0,x_1)$ and $E_R(x_{n-1},x_0)$, so there are no cycles.

To complete the proof, suppose that $X$ is a connected subgraph. Then any finite connected subgraph of $X$ is   $\d$-closed in $\H_\d$. Suppose $X$ is not   $\d$-closed. First assume that  $X$ is not gated. Then there exist $a\in\H_\d\setminus X$ for which there is no gate in  $X$. Let $x\in X$ be such that $d(a,x)$ is minimal. Then $x$ is unique: if $y\in X$ with $d(a,x)=d(a,y)$, then the convex closure of $\{x,y\}$ in $X$ is finite and not gated, contradicting the previous paragraph. Furthermore for any $y\in X$ we have $d(a,y)=d(a,x)+d(x,y)$ as this holds in the convex closure of  $\{x,y\}$ in $X$.
Now let $a,b\in \H_\d\setminus X$. Then $d(g_X(a),g_X(b))>\d$ and since the convex closure of $\{g_X(a),g_X(b)\}$ is   $\d$-closed, it follows that $d(a,b)$ is as required.

\end{proof}

\begin{proposition}\label{prop:trans not prim}
$\Isom(\H_\d)$ acts transitively, but not primitively on $\H_\d$. In fact, there exist uncountably many non-trivial equivalence relations on $\H_\d$ invariant under $\Isom(\H_\d)$.
\end{proposition}

\begin{proof}
Transitivity of the action 
follows directly from the $\leq$-homogeneity and was already noted above.

To see that the action is not primitive, let  $R\subset\G_{>\d}$ and define the equivalence relation $T_R$ on $\H_\d$ where $T_R(x,y)$ if and only if $x$ and $y$ lie in the same $E_R$-connected component. Clearly this is invariant under the action of $\Isom(\H_\d)$. Since there are uncountably many subsets $R$, the claim follows.
\end{proof}

\begin{definition}\label{def:T_R}
For $R\subset\G$ we call the $E_R$ connected components of $\H_\d$ a $T_R$-tree.
\end{definition}

\begin{remark}\label{rem: trees are closed}
Note that each $T_R$-tree $X$ is geodetically embedded, i.e. if $(x_0,x_1,\ldots, x_n)$ is a path in a $T_R$-tree, then $(x_0,x_1,\ldots, x_n)$ is geodetic. Since every finite convex subgraph of $X$ is   $\d$-closed in $\H_\d$ by the proof of Lemma~\ref{lem:tree}, it follows in particular, that $X$ is   $\d$-closed in $\H_\d$.
\end{remark}

 Hence we have:

\begin{corollary}\label{cor:no bounded isom}
No nontrivial element of $\Aut(\H_\d)$ has bounded displacement.
\end{corollary}
\begin{proof}
If $g\in\Aut(\H_\d)$ has bounded displacement, then so does every commutator $[g,h]$ for $h\in\Aut(\H_\d)$. By taking an appropriate commutator we may assume that $g$ fixes a point $x$ and acts nontrivially on the $T_R$-tree containing $x$. But no nontrivial isometry of $T_R$ has bounded displacement, proving the corollary.
\end{proof}

\begin{corollary}\label{cor:no fixed point}
Suppose $\a\in\Aut(\H_\d)$ acts as a translation on some $T_R$-tree $X$ for some $R$.  Then $\a$ has no fixed point. In particular, any $\a\in\Aut(\H_\d)$ for which there exists $x\in\H_\d$ such that $\{x,\a(x)\}$ is   $\d$-closed and $d(x,\a(x))>\d$ is fixed point free.
\end{corollary}
\begin{proof}
 Note that $X$ is   $\d$-closed in $\H_\d$. Let $c\in\H_\d\setminus X$ and $g=g_X(c)$. Then $d(c,g)=d(\a(c),\a(g))<d(c,\a(g))$ by Remark \ref{rem:unique gate} and hence $c\neq\a(c)$.

If $\{x,\a(x)\}$ is   $\d$-closed and $d=d(x,\a(x))>\d$, then $\a$ leaves the $T_R$-tree $X$  containing $x$ for $R=\{d\}$ invariant and acts on $X$ as a translation. So the result from the first part.
\end{proof}

\begin{proposition}\label{prop:not dense}
$G=\Aut(\H_\d)$ has no dense conjugacy class.
For any finite set $B$ the pointwise stabilizer $G_{\cl(B)}$ in $G$ has a dense conjugacy class.
\end{proposition}
\begin{proof}
We use the criterion from  \cite{TZ}, Lemma 2.8: $G$ has a dense conjugacy class if and only if for any finite tuples $x, y, a, b$ such that $\cl(x)\cong\cl(y)$  and $\cl(a)\cong\cl(b)$ there are tuples $x',y'$ such that $\cl(x'y')\cong\cl(xy)$ and $\cl(x'a)\cong\cl(y'b)$. 

To see that this critrion is violated, let $(x_0,x_1,x_2)$ be a path in a $T_{2\d}$-tree and put $x=\cl(x)=(x_0,x_1), y=\cl(y)=(x_1,x_2)$ and let $a=b\in\H_\d$ be an arbitrary point.
Suppose there are tuples $x',y'$ such that $\cl(x'y')\cong\cl(xy)$ and $\cl(x'a)\cong\cl(y'a)$. Then $x'=(x_0',x_1'), y'=(x_1',x_2')$ for a path
  $(x_0',x_1',x_2')$ in some $T_{2\d}$-tree and by $\leq_\d$-homogeneity there is some $g\in G$ such that $g(x'a)=(y'a)$. Thus, $g$ acts as a translation on the $T_2\d$-connected component containing $x_1'$ and hence $g$ has no fixed point by Corollary~\ref{cor:no fixed point}, a contradiction. 
  
  This shows that $G$ does not have a dense conjugacy class. The fact that  $G_{\cl(B)}$  has a dense conjugacy class for any finite subset $B$ follows from \cite{TZ}, Lemma 2.8 using the fact that $\H_\d$ has a local stationary independence relation.

\end{proof}

\section{Final remarks and open questions}

In the previous sections we tried to point out similarities and dissimilarities between the universal $\d$-hyperbolic spaces constructed here and the Urysohn space on the one hand and regular trees on the other hand.

In light of the fact that the isometry group of a regular tree and the isometry group of the Urysohn space modulo the normal subgroup of bounded  displacement are simple groups, it is natural to ask:
\begin{itemize}
\item Is the isometry group of $\H_\d$ topologically simple, or even abstractly simple?
\end{itemize}
 
Since $\H_\d$ has no isometries of bounded displacement, it seems likely that the answer is `yes'. However, the methods from \cite{TZ} do not apply here directly since $\Isom(\H_\d)$ has no dense conjugacy class. In the case of trees and more generally right-angled buildings, the automorphism groups also miss a dense conjugacy class. However, the methods from \cite{TZ} can be applied to stabilizers of maximal flags because these stabilizers have dense conjugacy classes,  and the same was shown above  for~$\H_\d$.

For right-angled buildings one can then deduce abstract simplicity of the automorphism group using the fact that stabilizers of partial flags form a sequence of subgroups maximal in each other. However, as pointed out above, the isometry group of $\H_\d$ does not act primitively. Hence point stabilizers are not maximal in this setting and it is harder to approach the simplicity problem.

Other natural questions address the model theoretic properties of these hyperbolic spaces:

\begin{itemize}

\item Is there a natural class of $\d$-hyperbolic spaces $\C$, $\d>0$, such that the limit structure is stable?
\end{itemize}

For $0$-hyperbolic spaces (and right-angled buildings) it was shown in \cite{pseudospace} that the theory is $\omega$-stable. This can be shown using the notion of independence defined above and verifying that it satisfies the necessary conditions for stability.
In general, one cannot expect arbitrary $\d$-hyperbolic spaces to be stable. It would be interesting to find a natural class of hyperbolic spaces with stable theory.

Note that an infinite tree of bounded height is $\omega$-categorical. Hence the following questions are natural given the existence of a bounded Urysohn space and $0$-hyperbolic spaces of bounded diameter:

\begin{itemize}
\item Is there a $\d$-hyperbolic analog of a bounded Urysohn space for $\d> 0$?

\item Is there a natural class of $\d$-hyperbolic spaces $\C$, for $\d> 0$, such that the limit structure is $\omega$-categorical?
\end{itemize}

\section{Acknowledgements}

The author wishes to thank the referee for constructive comments and Peter Kramer for providing the diagram for Theorem~\ref{thm: no universal-homogeneous dh}.
 
This research was partially funded through the Cluster of Excellence by the
German Research Foundation (DFG) under Germany’s Excellence Strategy EXC
2044–390685587, Mathematics M\"unster: Dynamics–Geometry–Structure and by
CRC 1442 Geometry: Deformations and Rigidity.


\begin{thebibliography}{999}


\bibitem{Bestvina} M. Bestvina, K. Bromberg, K. Fujiwara, Constructing group actions on quasi-trees and applications to mapping class groups. Publ. Math. Inst. Hautes Études Sci. 122 (2015), 1–64. 

\bibitem{DK} C. Druţu, M. Kapovich, \emph{Geometric group theory}.
With an appendix by Bogdan Nica. American Mathematical Society Colloquium Publications, 63. American Mathematical Society, Providence, RI, 2018. xx+819 pp.

 \bibitem{Gao} S. Gao, Su, Ch.  Shao,  Polish ultrametric Urysohn spaces and their isometry groups.
 Topology Appl.  158  (2011),  no. 3, 492--508.
 
 \bibitem{Hr} E. Hrushovski, A new strongly minimal set, Annals of Pure and Applied Logic,
vol. 62 (1993), no. 2, pp. 147--166, Stability in model theory, III.

\bibitem{Ishiki1} Y. Ishiki,  Constructions of Urysohn universal ultrametric spaces.
 p-Adic Numbers Ultrametric Anal. Appl.  15  (2023),  no. 4, 266--283.

\bibitem{Ishiki2} Y. Ishiki,  Uniqueness and homogeneity of non-separable Urysohn universal ultrametric spaces.
 Topology Appl.  342  (2024), Paper No. 108762, 11 pp.
		

\bibitem{Katetov} M. Katetov, On universal metric spaces, in: General Topology and Its Relations to Modern Analysis and Algebra, VI, Prague, 1986, in: Res. Exp. Math., vol. 16, Heldermann, Berlin, 1988, pp. 323–330.


\bibitem{Lemin1} V. A.  Lemin,   On a universal ultrametric space.
 Topology Appl.  103  (2000),  no. 3, 339--345.


\bibitem{Lemin2} V.A. Lemin, On metrically universal ultrametric spaces $LV_\tau$ and $LW_\tau$.
 Ultrametric functional analysis (Nijmegen, 2002), 
 191--205, Contemp. Math., 319, Amer. Math. Soc., Providence, RI,  2003.  
 
\bibitem{TZ} K. Tent, M. Ziegler, On the isometry group of the Urysohn space, Journal of the London Mathematical Society 2012; doi: 10.1112/jlms/jds027. 

\bibitem{TZmodel theory} K. Tent, M. Ziegler, \emph{A course in model theory}. Lecture Notes in Logic, 40. Association for Symbolic Logic, La Jolla, CA; Cambridge University Press, Cambridge, 2012

\bibitem{pseudospace} K. Tent, The free pseudospace is N-ample, but not (N+1)-ample. J. Symb. Log. 79 (2014), no. 2, 410--428. 

\bibitem{Wan} Zh. Wan, A novel construction of Urysohn universal ultrametric space via the Gromov-Hausdorff ultrametric.
 Topology Appl.  300  (2021), Paper No. 107759, 10 pp.
		

\end{thebibliography}
\end{document}